\documentclass[12pt]{amsart}
\usepackage{amsmath,amsfonts,amssymb}
\numberwithin{equation}{section}

\newtheorem{theorem}{Theorem}[section]
\newtheorem{lemma}[theorem]{Lemma}
\newtheorem{corollary}[theorem]{Corollary}
\newtheorem{proposition}[theorem]{Proposition}

\DeclareMathOperator{\Det}{Det} \DeclareMathOperator{\Tr}{Tr}

 \DeclareMathOperator{\Rea}{Re}

\DeclareMathOperator{\Arg}{Arg}

\title [Complex Weyl symbols...]{ Complex Weyl symbols of metaplectic operators: An elementary approach}
\author {Benjamin Cahen}
\address{Universit\'e de Lorraine, Site de Metz, UFR-MIM,
D\'epartement de math\'ematiques,
B\^atiment A,
3 rue Augustin Fresnel, BP 45112,
57073 METZ Cedex 03, France.}
\email{benjamin.cahen@univ-lorraine.fr}


\begin{document}

\maketitle

\begin{abstract} We give explicit formulas for the Berezin symbols and the complex Weyl symbols of the metaplectic representation operators by using the holomorphic representations of the Jacobi group. Then we recover some known formulas for the symbols of the metaplectic operators in the classical Weyl calculus, in particular for the classical Weyl symbol of the exponential of an operator whose Weyl symbol is a quadratic form.
\end{abstract}

\vspace{1cm}

\section {Introduction} \label{sec:intro}
The metaplectic representation (also called oscillator representation or Weil representation) is
a projective unitary representation of the symplectic group $Sp(n,{\mathbb R})$ which was first investigated by I. E. Segal,
D. Shale and A. Weil, see for instance \cite{Fo} and its references. The metaplectic representation plays an important role in very different aeras of mathematics such as number theory (automorphic forms) and mathematical physics (quantum mechanics).

For the computations, it is sometimes convenient to realize $Sp(n,{\mathbb R})$ as a subgroup $S$ of $SU(n,n)$, see \cite{Fo}, p. 175.  The group $S$ acts naturally
on the $(2n+1)$-dimensional
(real) Heisenberg group $H_n$ and then on the generic representations (the non-degenerated unitary irreducible representations) of $H_n$ on the Fock space $\mathcal F$.
Let $k\cdot \rho$ denote the action of 
$k\in S$ on the generic representation $ \rho$ of $H_n$.
Then $k\cdot  \rho$
and $ \rho$ are unitarily equivalent representations and there exists a unitary operator $\sigma(k)$ on $\mathcal F$ (defined up to a unit scalar) such that 
\begin{equation}\label{eq:inter}(k\cdot  \rho)(h)\sigma(k)=\sigma(k) \rho(h)\end{equation}
for each $h\in H_n$. The map $\sigma$ is thus the metaplectic representation of $S$.

One can find in the literature various methods to construct the metaplectic representation of $S$, that is, to obtain explicit formulas for $\sigma(k)$, $k\in S$,
see for instance \cite{Fo}  and \cite{KVe}. However, as mentioned in \cite{Ne}, p. 533, the most direct method is to use the holomorphic representations of the (multi-dimensional) Jacobi group $G:=H_n\rtimes S$. Here we apply this method in full details in order to get explicit expressions for the kernels of the metaplectic operators $\sigma(k)$ and then for the Berezin symbols of  $\sigma(k)$. Note that this method also works for the harmonic representation of 
$SU(p,q)$, see \cite{CaHarm}.

The complex Weyl calculus $W_0$ is the correspondence between operators on $\mathcal F$ and functions on ${\mathbb C}^n$ obtained by translating the usual Weyl correspondence
(see \cite{Fo}, \cite{Ho1}) by means of the Bargmann transform. It is known that $W_0$
is the unitary component in the polar decomposition of the Berezin correspondence on $\mathcal F$, see \cite{Luo}, \cite{CaRiv}.

In the present paper, we give explicit formulas for the complex Weyl symbols of the metaplectic operators. More precisely, we compute $W_0(\sigma(k))$ for $k\in S$ and $W_0(d\sigma(X))$
for $X$ in the Lie algebra $\mathfrak s$ of $S$. Our method is quite elementary, let us describe it briefly. For any operator $A$ on $\mathcal F$, $W_0(A)$ can be expressed by an integral formula involving the kernel of $A$. This allows us to reduce the computation of  $W_0(d\sigma(X))$ to that of some Gaussian integral. 

As an immediate consequence, we recover some known formulas for $W_1(\sigma'(g))$, $g\in Sp(n,{\mathbb R})$ and 
 $W_1(d\sigma'(X))$, $X\in sp(n,{\mathbb R})$ where $\sigma'$ denotes the metaplectic representation of $Sp(n,{\mathbb R})$ and $W_1$ the inverse map of the classical Weyl correspondence, \cite{Fo}, \cite{CR1}. In particular, since $W_1(d\sigma'(X))$ is a quadratic form on ${\mathbb R}^{2n}$ for each $X\in sp(n,{\mathbb R})$, we recover a formula of \cite{Ho2} for the Weyl symbol of the exponential of an operator having a quadratic form as Weyl symbol. A similar formula for the star exponential
of a quadratic form for the Moyal star product can be found in \cite{BM}. Note that these formulas were then established in  \cite{Ho2} and \cite{BM} by solving some differential systems, see also \cite{CR2}. Note also that the study of the Weyl symbol of the function of an operator whose Weyl symbol is quadratic is still a subject of active research, see for instance
\cite{Der2} and references therein.

The plan of this paper is as follows. In Section \ref{sec:2}, we review some generalities
about the non-degenerated unitary irreducible representations of $H_n$ on the Fock space
and about the Berezin correspondence. In Section \ref{sec:3}, we introduce the complex Weyl correspondence and we emphasize its connection with the Berezin correspondence and also with
the classical Weyl correspondence. In Section \ref{sec:4}, we consider the metaplectic representation $\sigma$ and we give a functional equation satisfied by the kernel of $\sigma(k)$ for $k\in S$.
Section \ref{sec:5} is devoted to a short presentation of the (multi-dimensional) Jacobi group
$G=H_n\rtimes S$ and its holomorphic representations \cite{Bern}. We deduce from the
formula for the Berezin symbol of a holomorphic representation operator  $\pi(h,k)$ given in
\cite{CaPad2} a functional equation for the kernel of $\pi(id,k)$ which is used in Section \ref{sec:6} to find formula for the kernel of $\sigma(k)$. From this, we derive in Section \ref{sec:7} a formula for $W_0(\sigma(k))$ ($k\in S$). Finally, we compute $W_0(d\sigma(X))$ ($X\in {\mathfrak s}$) in Section \ref{sec:8} and we relate the results to those of \cite{Ho2} and \cite{BM}.

\section{Heisenberg group: Berezin quantization} \label{sec:2}

In this section, we first review some general facts on the
the Bargmann-Fock
model for the unitary irreducible (non-degenerated)
representations of the Heisenberg group, see \cite{Fo}. We  follow the presentation
of \cite{CaTo}, see also \cite{CaTs} .

For each $z,\,w \in {\mathbb C}^{n}$, we denote $zw:=\sum_{k=1}^nz_kw_k$. For each $z, z',w,w'\in {\mathbb C}^{n}$, let
\begin{equation*}\omega
((z,w),(z',w'))=\tfrac{i}{2}(zw'-z'w).
\end{equation*}

Then the $(2n+1)$-dimensional real Heisenberg group is $$H:=\{((z,{\bar
z}),c)\,:\,z\in {\mathbb C}^n, c\in {\mathbb R}\}$$ endowed with the
multiplication law
\begin{equation*}((z,{\bar
z}),c)\cdot ((z',{\bar z'}),c')=((z+z',{\bar z}+{\bar
z'}),c+c'+\tfrac{1}{2}\omega ((z,{\bar z}),(z',{\bar
z'}))).\end{equation*}

Let $\lambda>0$. By the Stone-von Neumann
theorem, there exists a unique (up to unitary equivalence) unitary
irreducible representation $\rho_{\lambda}$ of $H_n$ whose restriction to the center
of $H_n$ is the character $(0,c)\rightarrow e^{i\lambda c}$
\cite{Tay}. 
The Bargmann-Fock realization of $\rho_{\lambda}$ is defined as follows \cite{Barg}. 

Let ${\mathcal F}_{\lambda}$ be the Hilbert space of all holomorphic functions $f$
on ${\mathbb C}^n$ such that \begin{equation*}\Vert f\Vert^2_{{\mathcal F}_{\lambda}}
:=\int_{{\mathbb C}^n} \vert f(z)\vert^2\, e^{-\lambda \vert
z\vert^2/2}\,d\mu_{\lambda} (z) <+\infty\end{equation*} where
 $d\mu_{\lambda}(z):=(2\pi
)^{-n}{\lambda}^n\,dm(z)$. Here $z=x+iy$ with $x$ and $y$ in ${\mathbb
R}^n$ and $dm(z):= dx\,dy$ is the standard Lebesgue measure on ${\mathbb
C}^n$.

Then 
\begin{equation*}({\rho}_{\lambda}(h)f)(z)=\exp \left(i\lambda c_0+\tfrac{\lambda}{2}{\bar z_0}z-\tfrac{\lambda}{4}\vert z_0\vert^2\right)\,f(z- z_0) \end{equation*}
for each $h=(z_0, c_0)\in H_n$ and $z\in {\mathbb C}^n$.

For each $z\in {\mathbb C}^n$, consider the \textit {coherent state}
$e_z(w)=\exp (\lambda{\bar z}w/2)$. Then we have the reproducing property
$f(z)=\langle f,e_z\rangle_{{\mathcal F}_{\lambda}}$ for each $f\in {\mathcal
F}_{\lambda}$.

We can introduce the Berezin
calculus on ${\mathcal F}_{\lambda}$ \cite{Be1}, \cite{Be2}, \cite{CaRiv}.
The Berezin (covariant) symbol of an operator $A$ on ${\mathcal
F}_{\lambda}$ is the
function $S_{\lambda}(A)$ defined on ${\mathbb C}^n$ by
\begin{equation*}S_{\lambda}(A)(z):=\frac {\langle A\,e_z\,,\,e_z\rangle_{{\mathcal F}_{\lambda}}}{\langle e_z\,,\,e_z\rangle_{{\mathcal F}_{\lambda}}} \end{equation*}
and the double Berezin symbol $s_{\lambda}$ is defined by
\begin{equation*}s_{\lambda}(A)(z,w):=\frac {\langle A\,e_w\,,\,e_z\rangle_{{\mathcal F}_{\lambda}}}{\langle e_w\,,\,e_z\rangle_{{\mathcal F}_{\lambda}}} \end{equation*}
for each $(z,w)\in {\mathbb C}^n \times {\mathbb C}^n$ such that 
$\langle e_w\,,\,e_z\rangle_{{\mathcal F}_{\lambda}}\not= 0$.

Note that $s_{\lambda}(A)(z,w)$ is holomorphic in the variable $z$ and anti-holomorphic in the variable $w$, then $s_{\lambda}(A)$ is determined by its restriction to the diagonal of ${\mathbb C}^n \times {\mathbb C}^n$, that is, by $S_{\lambda}(A)$. Moreover, the operator $A$ can be recovered from $s_{\lambda}(A)$ as follows. We have

\begin{align*}
A\,f(z)&=\langle A\,f\,,\,e_z \rangle_{{\mathcal F}_{\lambda}}  
  =\langle f\,,\,A^{\ast}\,e_z\rangle_{{\mathcal F}_{\lambda}}  \\
&=\int _{{\mathbb C}^n}\,f(w)\overline {A^{\ast}\,e_z(w)}\,e^{-\lambda \vert
w\vert^2/2}\,d\mu_{\lambda} (w) \\
&=\int _{{\mathbb C}^n}\,f(w)\overline {\langle A^{\ast}\,e_z,e_w\rangle}\,e^{-\lambda \vert
w\vert^2/2}\,d\mu_{\lambda} (w) \\
&=\int _{{\mathbb C}^n} \,f(w)\,s_{\lambda}(A)(z,w)\langle e_w,e_z\rangle\,e^{-\lambda \vert
w\vert^2/2}\,d\mu_{\lambda}(w).\\ 
\end{align*}
In particular, we see that  the map $A\rightarrow S_{\lambda}(A)$ is injective and that the
kernel of $A$ is the function
\begin{equation} k_A(z,w)=\langle Ae_w,e_z\rangle_{{\mathcal F}_{\lambda}}=s_{\lambda}(A)(z,w)\langle e_w,e_z\rangle_{{\mathcal F}_{\lambda}}.\end{equation}

The map $S_{\lambda}$ is a bounded operator from the
space ${\mathcal L}_2({\mathcal F}_{\lambda})$ of all Hilbert-Schmidt operators
on ${\mathcal F}_{\lambda}$ (endowed with the
Hilbert-Schmidt norm) to $L^2({\mathbb C}^n,\mu_{\lambda})$  which
is one-to-one and has dense range \cite{UU}. Let us introduce the Berezin transform which will be needed later. Let $S_{\lambda}^{\ast}$ be the adjoint operator of $S_{\lambda}$.
Then
the Berezin transform is the operator $B_{\lambda}$
on $L^2({\mathbb C}^n,\mu_{\lambda})$ defined by $B_{\lambda}:=S_{\lambda}S_{\lambda}^{\ast}$.
We have the integral formula
\begin{equation*}(B_{\lambda}f)(z)=\int_{{\mathbb C}^n}\,f(w)
\,e^{ -\lambda\vert z-w\vert^2/2}\,d\mu_{\lambda}(w),\end{equation*}
see \cite{Be1}, \cite{Be2}, \cite{UU}.
Note also that we have $B_{\lambda}=\exp (\Delta/2\lambda)$ where
$\Delta=4\sum_{k=1}^n\partial^2/\partial z_k\partial {\bar z}_k$,
see \cite{UU}, \cite{Luo}.

\section{Complex Weyl correspondence for Heisenberg group} \label{sec:3}

The complex Weyl correspondence can be constructed from a \textit{Stratonovich-Weyl quantizer}
see \cite{St}, \cite{GB}, \cite{CaTo} and \cite{AU1}, Example 2.2 and Example 4.2.

Let $R_0$ be the parity operator on ${\mathcal F}_{\lambda}$ defined by
\begin{equation*}(R_0f)(z)=2^nf(-z). \end{equation*}

Then we define
\begin{equation*}\Omega_0(z):=\rho_{\lambda}((z,{\bar z}),0)R_0\rho_{\lambda}((z,{\bar z}),0)^{-1} \end{equation*} for each $z\in {\mathbb C}^n$. 
By an easy computation we get
\begin{equation}\label{eq:om}(\Omega_0(z)f)(w)=2^n\exp \left({\lambda}(w{\bar z}-\vert z\vert^2)\right) f(2z-w) \end{equation}
for each $z, w\in {\mathbb C}^n$ and $f\in {\mathcal F}_{\lambda}$. 
The map $\Omega_0$ is called a \textit{Stratonovich-Weyl quantizer}.
For each trace-class operator $A$ on ${\mathcal F}_{\lambda}$, we define 
\begin{equation*} W_0(A)(z):=\Tr(A\Omega_0(z))\end{equation*}
for each $z \in {\mathbb C}^n$.

Recall that we denote by $k_A$ the kernel  
of the trace-class (or more generally Hilbert-Schmidt) operator $A$ on 
${\mathcal F}_{\lambda}$, see Section \ref{sec:2}.
We have the following result, see \cite{CaTs}, \cite{CaTo}, \cite{AU1}.

\begin{proposition} \label{prop:intW} For each trace-class operator $A$ on ${\mathcal F}_{\lambda}$
and each $z \in {\mathbb C}^n$, we have
\begin{equation} \label{eq:intW}
W_0(A)(z)=2^n\int_{{\mathbb C}^n}k_A(w,2z-w)\exp \left({\lambda}\left(-z{\bar z}+z{\bar w}-\tfrac{1}{2}w{\bar w}\right)\right) d\mu_{\lambda}(w). 
\end{equation} and, equivalently, on a more symmetric form
\begin{equation} \label{eq:intWsym}
W_0(A)(z)=2^n\int_{{\mathbb C}^n}k_A(z+w,z-w)\exp \left(\tfrac{\lambda}{2}\left(-z{\bar z}-w{\bar w}+z{\bar w}-{\bar z}w\right)\right) d\mu_{\lambda}(w).\end{equation}
\end{proposition}

These integral formulas allow us to extend $W_0$ to operators
on ${\mathcal F}_{\lambda}$  which are not necessarily trace-class, for instance Hilbert-Schmidt operators.  It is known that $W_0:{\mathcal L}_2({\mathcal F}_{\lambda})\rightarrow L^2({\mathbb C}^n,\mu_{\lambda})$ is the unitary part in the polar decomposition of $S_{\lambda}$, that is we have $S_{\lambda}=B_{\lambda}^{1/2}W_0$, see  \cite{Luo}, Theorem 6, \cite{CaRiv}, \cite{CaTs}. 

Now, with the aim of linking $W_0$ to the classical Weyl correspondence, we consider
another realization of the unitary irreducible representation of $H_n$ with central character
$(0,c)\rightarrow e^{i\lambda c}$, namely the Schr\"odinger representation $\rho'_{\lambda}$
defined on $L^2({\mathbb R}^n)$ by
\begin{equation*}(\rho'_{\lambda}((a+ib,a-ib),c)\phi)(x)
=\exp \left(i\lambda (c-bx+\tfrac{1}{2}ab)\right)\,\phi(x-a) \end{equation*}
for each $a, b, x \in {\mathbb R}^n$.

An (unitary) intertwining operator between $\rho_{\lambda}$ and $\rho'_{\lambda}$ is the Bargmann
transform $\mathcal B: L^2({\mathbb R}^n)\rightarrow {\mathcal F}_{\lambda}$  defined by
\begin{equation*}({\mathcal B}f)(z)=\left(\tfrac{\lambda}{\pi}\right)^{n/4}\,\int_{{\mathbb R}^n}\,\exp
\left(-\tfrac{\lambda}{4}z^2+\lambda zx-\tfrac{\lambda}{2}x^2\right)
\,\phi(x)\,dx,\end{equation*} 
see \cite{Fo}, \cite{CaRiv}.

Starting from the parity operator $R_1$ on $ L^2({\mathbb R}^n)$ defined by
\begin{equation*}(R_1\phi)(x)=2^n\phi(-x), \end{equation*}
we can define the Stratonovich-Weyl quantizer $\Omega_1$ on ${\mathbb R}^{2n}$ by
\begin{equation*}\Omega_1(a,b):=\rho'_{\lambda}((a+ib,a-ib),0)R_1\rho'_{\lambda}((a+ib,a-ib),0)^{-1} \end{equation*} or, equivalently, by
\begin{equation}(\Omega_1(a,b)\phi)(x)=2^n\exp \left(2i{\lambda}b(a-x)\right) \phi(2a-x) \end{equation}
for each $\phi \in L^2({\mathbb R}^n)$.
Then, for each trace-class operator $A$ on $L^2({\mathbb R}^n)$, we define the function 
$W_1(A)$ on ${\mathbb R}^{2n}$ by 
\begin{equation*} W_1(A)(x,y):=\Tr(A\Omega_1(x,y))\end{equation*}
for each $x, y \in {\mathbb R}^n$.

On the other hand, recall that the classical Weyl correspondence on ${\mathbb R}^{2n}$ is defined as follows \cite{Fo}, \cite{Ho1}.
For each function $f$ in the Schwartz space ${\mathcal S}({\mathbb
R}^{2n})$, we define the operator $ {\mathcal W}(f)$ acting on the Hilbert
space $L^2({\mathbb R}^{n})$ by \begin{equation}\label{eqcalWeyl} ({\mathcal W}(f)\phi)
(x)={(2\pi)}^{-n}\,\int_{{\mathbb R}^{2n}}\, e^{i yt} f( x+\tfrac{1}{2}y,
t)\,\phi (x+y)\,dy\,dt.
\end{equation}

Consider the Fourier transform ${\mathcal F}_2f$ of $f\in {\mathcal S}({\mathbb
R}^{2n})$ with respect to the second variable
\begin{equation*}({\mathcal F}_2f)(x,y)={(2\pi)}^{-n/2}\int_{{\mathbb R}^{n}}e^{-i yt
}f(x,t)\,dt.\end{equation*}
Then we can write
\begin{equation*}({\mathcal W}(f)\phi)
(x)={(2\pi)}^{-n/2}\int_{{\mathbb R}^{n}} ({\mathcal F}_2f)(\tfrac{1}{2}(x+y),x-y)\phi(y)dy\end{equation*}
and we can see that the kernel of ${\mathcal W}(f)$ is
\begin{equation*}k_{{\mathcal W}(f)}(x,y)={(2\pi)}^{-n/2}({\mathcal F}_2f)(\tfrac{1}{2}(x+y),x-y).\end{equation*}

\begin{lemma} For each $f\in {\mathcal S}({\mathbb
R}^{2n})$ such that ${\mathcal W}(f)$ is trace-class we have
\begin{equation*}\Tr(\Omega_1(a,b){\mathcal W}(f))=f(a,\lambda b)\end{equation*}
for each $a, b \in {\mathbb R}^{n}$.\end{lemma}

\begin{proof} Let $f\in {\mathcal S}({\mathbb
R}^{2n})$ and $a, b \in {\mathbb R}^{n}$. Then we have
\begin{align*} (\Omega_1(a,b){\mathcal W}(f)&\phi)(x)=
2^n\exp \left(2i{\lambda}b(a-x)\right)\,({\mathcal W}(f)\phi)(2a-x)\\
=&2^n\exp \left(2i{\lambda}b(a-x)\right)\,\int_{ {\mathbb R}^{n}}k_{{\mathcal W}(f)}(2a-x,y)\phi(y)\,dy.
\end{align*}
Thus the kernel of $\Omega_1(a,b){\mathcal W}(f)$ is 
\begin{equation*}k'(x,y):=2^n\exp \left(2i{\lambda}b(a-x)\right)k_{{\mathcal W}(f)}(2a-x,y)\end{equation*} and by Mercer's theorem we have
\begin{align*}\Tr(\Omega_1(a,b)&{\mathcal W}(f))=\int_{ {\mathbb R}^{n}}k'(x,x)\,dx\\
&=2^n\,\int_{ {\mathbb R}^{n}}\exp \left(2i{\lambda}b(a-x)\right)k_{{\mathcal W}(f)}
(2a-x,x)\,dx\\
&={(2\pi)}^{-n/2}2^n\,\int_{ {\mathbb R}^{n}}\exp \left(2i{\lambda}b(a-x)\right)({\mathcal F}_2f)(a,2a-2x)\,dx\\
&={(2\pi)}^{-n/2}\,\int_{ {\mathbb R}^{n}}\exp \left(i{\lambda}bx\right)({\mathcal F}_2f)(a,x)\,dx\\
&=f(a,\lambda b),
\end{align*}
by the Fourier inversion theorem.
\end{proof}

As an immediate consequence of this lemma, we get the following proposition.
\begin{proposition} Let $f\in {\mathcal S}({\mathbb R}^{2n})$ such that ${\mathcal W}(f)$
is trace-class. Then, for each $x,y \in {\mathbb R}^n$, we have
\begin{equation*}W_1({\mathcal W}(f))(x,y)=f(x,\lambda y).
\end{equation*}
\end{proposition}
If, in particular,  we take $\lambda =1$, we see that $W_1$ and $\mathcal W$ are inverse
to each other. 

We can now specify the connection between $W_0$ and $W_1$, hence between $W_0$ and $\mathcal W$.
\begin{proposition} \label{propconnect} For each trace-class operator $A$ on $L^2({\mathbb R}^{n})$ and each
$a, b\in {\mathbb R}^{n}$, we have
\begin{equation*}W_1(A)(a,b)=W_0({\mathcal B}A{\mathcal B}^{-1})(a+ib).\end{equation*}
\end{proposition}

\begin{proof} We can easily verify that ${\mathcal B}R_1=R_0{\mathcal B}$. Let
$A$ be a trace-class operator on $L^2({\mathbb R}^{n})$ and let
$a, b\in {\mathbb R}^{n}$. Recall that $\mathcal B$ intertwines $\rho_{\lambda}$ and 
$\rho'_{\lambda}$. Then we have
\begin{align*}
W_1(A)(a,b)=&\Tr(A\Omega_1(a,b))=\Tr(A\rho'_{\lambda}((a+ib,a-ib),0)R_1\rho'_{\lambda}((a+ib,a-ib),0)^{-1})\\
=&\Tr(A\rho'_{\lambda}((a+ib,a-ib),0){\mathcal B}^{-1}R_0{\mathcal B}\rho'_{\lambda}((a+ib,a-ib),0)^{-1})\\
=&\Tr(A{\mathcal B}^{-1}\rho_{\lambda}((a+ib,a-ib),0)R_0\rho_{\lambda}((a+ib,a-ib),0)^{-1}{\mathcal B})\\
=&\Tr ({\mathcal B}A{\mathcal B}^{-1}\Omega_0(a+ib))=W_0({\mathcal B}A{\mathcal B}^{-1})(a+ib).
\end{align*}
\end{proof}

Of course, Proposition \ref{propconnect} can be extended to operators which are not
necessarily of trace-class.

\section{The metaplectic representation} \label{sec:4}

Here we consider the group $S:=Sp(n,{\mathbb C})\cap SU(n,n)$
which is isomorphic to $Sp(n,{\mathbb R})$ via the map $M\rightarrow UMU^{-1}$
where $U:=\left(\begin{smallmatrix} I_n&iI_n\\
I_n&-iI_n\end{smallmatrix}\right)$, see \cite{Fo}, p. 175.
Then $S$
consists of all matrices
\begin{equation*}k=\begin{pmatrix} P&Q\\
{\bar Q}&{\bar P}
\end{pmatrix}, \quad P,Q\in M_n({\mathbb C}),\quad PP^{\ast}-QQ^{\ast}=I_n,\quad PQ^t=QP^t.
\end{equation*}
Note that we also have
\begin{equation*}P^{\ast}P-Q^t{\bar Q}=I_n,\quad \quad P^{\ast}Q=Q^t{\bar P}.
\end{equation*}
This implies, in particular, that $P^{-1}Q$ and ${\bar Q}P^{-1}$ are symmetric.

The group $S$ acts on ${\mathbb C}^n$ by 
\begin{equation*}k(z,{\bar z})=(Pz+Q{\bar z},{\bar Q}z+{\bar P}{\bar z})\end{equation*}
where $k=\left(\begin{smallmatrix} P&Q\\
{\bar Q}&{\bar P}\end{smallmatrix}\right)$. Then $\omega$ is invariant for the action of $S$,
that is, we have
\begin{equation*}\omega (k(z,{\bar z}),k(z',{\bar z'}))=\omega ((z,{\bar z}),(z',{\bar z'}))
\end{equation*} for each $z,z'\in {\mathbb C}^n$. We will denote
$kz=Pz+Q{\bar z}$.

Thus $S$ also acts on $H_n$ by $k\cdot ((z,{\bar z}),c)=(k(z,{\bar
z}),c)$. 

Let $\lambda >0$ and $\rho:=\rho_{\lambda}$. For each $k\in S$, we define
$\rho_k$ by $\rho_k(h):=\rho(k\cdot h)$ for each $h\in H_n$. Since $k\in S$ acts on $H_n$
as a group isomorphism, we see that $\rho_k$ is also a generic representation of $H_n$.
Moreover, for each $h=((0,0),c)$
in the center of $H_n$, we have $\rho_k(h)=\rho(h)$. Hence, by the Stone-von Neumann
theorem, $\rho_k$ and $\rho$ are unitarily equivalent, that is, there exists a unitary operator
$A_k$ of ${\mathcal F}_{\lambda}$ (defined up to a unit complex number) such that
\begin{equation}\label{eq:Ak}\rho_k(h)=A_k\rho (h)A_k^{-1}\end{equation}
for each $h\in H_n$. 

Now we express Eq. \ref{eq:Ak} in terms of kernels.

\begin{proposition} \label{prop:relatker1} Let $k\in S$. Then the operator $A_k$ 
on ${\mathcal F}_{\lambda}$ satisfies Eq. \ref{eq:Ak} if and only if
its kernel $b_k(z,w):=k_{A_k}(z,w)$ satisfies the relation
\begin{equation}\label{eq:relatker1}
\exp  \left( -\tfrac{\lambda}{4}\vert kz_0\vert^2+\tfrac{\lambda}{2}\overline{kz_0}z\right)b_k(z-kz_0,w)
=\exp \left( -\tfrac{\lambda}{4}\vert z_0\vert^2-\tfrac{\lambda}{2}\bar{w}z_0\right)
b_k(z, w+z_0)
\end{equation}
for each $z_0,z$ and $w$ in ${\mathbb C}^n$.
\end{proposition}

\begin{proof} Let $k\in S$ and $h=((z_0,{\bar z}_0),0)\in H_n$. We have
\begin{equation*}(A_kf)(z)=\int_{{\mathbb C}^n}b_k(z, w)f(w)\,
e^{-\lambda \vert
w\vert^2/2}\,d\mu_{\lambda} (w).
\end{equation*} Then, on the one hand, we get
\begin{align*}\rho_{\lambda}(k\cdot h)(A_kf)(z)=&\exp  \left( -\tfrac{\lambda}{4}\vert kz_0\vert^2+\tfrac{\lambda}{2}\overline{kz_0}z\right)\\
\times & \int_{{\mathbb C}^n}b_k(z-kz_0, w)f(w)\,
e^{-\lambda \vert
w\vert^2/2}\,d\mu_{\lambda} (w).
\end{align*}
On the other hand, we have
\begin{equation*}(A_k\rho_{\lambda} (h)f)(z)=
 \int_{{\mathbb C}^n}b_k(z, w)\exp  \left( -\tfrac{\lambda}{4}\vert z_0\vert^2+\tfrac{\lambda}{2}\bar{z_0}w\right) f(w-z_0)\,
e^{-\lambda \vert
w\vert^2/2}\,d\mu_{\lambda} (w).
\end{equation*}
By performing the change of variables $w\rightarrow w+z_0$ in this integral and then writing
that the kernels of $\rho_{\lambda} (k\cdot h)A_k$ and $A_k\rho_{\lambda} (h)$ are the same, we obtained the desired equation.
\end{proof}

\section{Holomorphic representations of the Jacobi group} \label{sec:5}

In this section, we essentially follow \cite{CaPad2} in which we contructed the
holomorphic representations of the Jacobi group by applying the general method of 
\cite{Ne}, Chap. XII (see also \cite{Bern}, \cite{CaPad}).

The
(multi-dimensional) Jacobi group is the semi-direct product $G:=H_n\rtimes S$ with respect to the action of $S$ on $H_n$ introduced in Section \ref{sec:4}. The elements of $G$ can be written as $((z,{\bar
z}),c,k)$ where $z\in {\mathbb C}^n$, $c\in {\mathbb R}$ and $k\in
S$. The multiplication of $G$ is given by
\begin{equation*}((z,{\bar
z}),c,k)\cdot ((z',{\bar z'}),c',k')=((z,{\bar z})+k(z',{\bar
z'}),c+c'+\tfrac{1}{2}\omega ((z,{\bar z}),k(z',{\bar
z'})),kk').\end{equation*} The complexification $G^c$ of $G$ is then
the semi-direct product $G^c=H_n^c\rtimes Sp(n,{\mathbb C})$ whose elements can be written as $((z,w),c,k)$ where $z,w\in {\mathbb C}^n$, $c\in {\mathbb C}$, $k\in
Sp(n,{\mathbb C})$ and the
multiplication of $G^c$ is obtained by replacing ${\bar z}$ and
${\bar z'}$ by $w$ and $w'$ in the preceding formula. 

Let ${\mathfrak g}$ and ${\mathfrak
g}^c$ the Lie algebras of $G$ and $G^c$. For each
$X=\bigl((z,w),c,\bigl(\begin{smallmatrix} A&B\\
C&-A^t
\end{smallmatrix}\bigr)\bigr)\in {\mathfrak g}^c$ we define
\begin{equation*}X^{\ast}=((-{\bar w},-{\bar z}),-{\bar c},
\bigl(\begin{smallmatrix} {\bar A}^t&-{\bar C}\\
-{\bar B}&-{\bar A}
\end{smallmatrix}\bigr)).\end{equation*} 
We denote by $g\rightarrow g^{\ast}$ the
involutive anti-automorphism of $G^c$ which is obtained by
exponentiating $X\rightarrow X^{\ast}$ to $G^c$. 

Let $K$ be the subgroup of $G$ consisting of all
elements $\bigl((0,0),c,\bigl(\begin{smallmatrix} P&0\\
0&{\bar P}\end{smallmatrix}\bigr)\bigr)$ where $c\in {\mathbb R}$ and $P\in U(n)$.
Let $P^+$ and $P^-$ be the subgroups of $G$ defined by
\begin{equation*}P^+=\left\{\left((y,0),0,\begin{pmatrix} I_n&Y\\
0&I_n\end{pmatrix}\right)\,:\,y\in {\mathbb C}^n, Y\in M_n({\mathbb
C}),Y^t=Y\right\}\end{equation*} and
\begin{equation*}P^-=\left\{\left((0,v),0,\begin{pmatrix} I_n&0\\
V&I_n\end{pmatrix}\right)\,:\,v\in {\mathbb C}^n, V\in M_n({\mathbb
C}),V^t=V\right\}\end{equation*} 
and let ${\mathfrak p}^+$ and ${\mathfrak p}^-$ be the Lie algebras of 
$P^+$ and $P^-$. For convenience, we denote by $a(y,Y)$ the element $\bigl((y,0),0,\bigl(\begin{smallmatrix} 0&Y\\
0&0\end{smallmatrix}\bigr)\bigr)$ of ${\mathfrak p}^+$.

We can verify that each element
$g=\bigl((z_0,w_0),c_0,\bigl(\begin{smallmatrix} A&B\\
C&D\end{smallmatrix}\bigr)\bigr)\in G^c$ has a
$P^+K^cP^-$-decomposition if and only if
$\Det(D)\not= 0$ and, in this case, we have
\begin{equation*}g=\left((y,0),0,\begin{pmatrix} I_n&Y\\
0&I_n\end{pmatrix}\right)\cdot\left((0,0),c,\begin{pmatrix} P&0\\
0&(P^t)^{-1}\end{pmatrix}\right)\cdot\left((0,v),0,\begin{pmatrix} I_n&0\\
V&I_n\end{pmatrix}\right)\end{equation*} where $y=z_0-BD^{-1}w_0$,
$Y=BD^{-1}$, $v=D^{-1}w_0$, $V=D^{-1}C$, $P=A-BD^{-1}C=(D^t)^{-1}$
and $c=c_0-\tfrac{i}{4}(z_0-BD^{-1}w_0)w_0$. We then denote by
$\zeta:\, P^+K^cP^-\rightarrow P^+$ and $\kappa:\, P^+K^cP^-\rightarrow
K^c$  the projections onto $P^+$- and $K^c$-components.

Consider the action (defined almost everywhere) of $G^c$ on ${\mathfrak p}^+$ defined as follows.
For $Z\in {\mathfrak p}^+$ and
$g\in G^c$ with $g\exp Z \in P^+K^cP^-$, we define the element
$g\cdot Z$ of ${\mathfrak p}^+$ by $g\cdot Z:=\log \zeta (g\exp Z)$.
We can verify that
the action of $g=\bigl((z_0,w_0),c_0,\bigl(\begin{smallmatrix} A&B\\
C&D\end{smallmatrix}\bigr)\bigr)\in G^c$ on $a(y,Y)\in {\mathfrak
p}^+$ is given by $g\cdot a(y,Y)=a(y',Y')$ where
$Y'=(AY+B)(CY+D)^{-1}$ and
\begin{equation*}y'=z_0+Ay-(AY+B)(CY+D)^{-1}(w_0+Cy). \end{equation*}
Consequently, we have
\begin{equation*}{\mathcal D}:=G\cdot 0=\{a(y,Y)\in {\mathfrak
p}^+\,:\,I_n-Y{\bar Y}>0\}\cong {\mathbb C}^n\times {\mathcal B}.\end{equation*}
where ${\mathcal B}:=\{Y\in M_n({\mathbb C})\,:Y^t=Y,\,I_n-Y{\bar Y}>0\}$.

Let $\chi$ be a unitary character 
of $K$ whose extension to $K^c$ is also denoted by $\chi$.
Following \cite{Ne}, we introduce the functions $K_{\chi}(Z,W):=\chi ( \kappa (\exp W^{\ast}\exp Z))^{-1}$ for
$Z,\,W\in {\mathcal D}$ and $J_{\chi}(g,Z):=\chi(\kappa (g\exp Z))$
for $g\in G$ and $Z\in {\mathcal D}$. We consider the Hilbert space
${\mathcal H}_{\chi}$ of all holomorphic functions $f$ on $\mathcal D$ such
that
\begin{equation*}\Vert f\Vert
_{\chi}^2:=\int_{\mathcal D}\,\vert f(Z)\vert ^2\,
K_{\chi}(Z,Z)^{-1}c_{\chi}d\mu (Z)<+\infty. \end{equation*}
Here the $G$-invariant measure $\mu$ on $\mathcal D$ is defined by
\begin{equation*}d\mu(Z)=\Det (1-Y{\bar Y})^{-(n+2)}\,d\mu_L(y,Y),\end{equation*} 
where $d\mu_L$ is the Lebesgue measure on $\mathcal D\cong {\mathbb C}^n\times {\mathcal B}$, see \cite{Ne}, p. 538, and 
the constant $c_{\chi}$ is defined by
\begin{equation*}c_{\chi}^{-1}=\int_{\mathcal
D}\,K_{\chi}(Z,Z)^{-1}\,d\mu(Z).\end{equation*}

Let us fix $\chi$ as follows. Let $\lambda >0$ and $m\in {\mathbb Z}$. Then,
for each $k=\bigl((0,0),c,\bigl(\begin{smallmatrix} P&0\\
0&{\bar P}
\end{smallmatrix}\bigr)\bigr)\in K$, we set $\chi(k)=e^{i\lambda c}(\Det P)^{m}$.

\begin{proposition} \label{prop:Jac} {\rm \cite{Ne}, \cite{CaPad2}} \begin{enumerate} \item Let $Z=a(y,Y)\in {\mathcal D}$ and $W=a(v,V)\in {\mathcal D}$. 
We have \begin{align*}K_{\chi}(Z,&W)=\Det (I_n-Y{\bar V})^{m}\\
&\times \exp \left( \tfrac{\lambda}{4}\left(2y(I_n-{\bar V}Y)^{-1}{\bar v}+y(I_n-{\bar V}Y)^{-1}{\bar V}y+{\bar v}Y(I_n-{\bar V}Y)^{-1}{\bar v}\right)\right). \end{align*}

\item ${\mathcal
H}_{\chi}\not= (0)$ if and only if $m+n+1/2<0$. In this case, ${\mathcal
H}_{\chi}$ contains the polynomials.

\item For each $g=\bigl((z_0,{\bar z_0}),c_0,\bigl(\begin{smallmatrix} P&Q\\
{\bar Q}&{\bar P}\end{smallmatrix}\bigr)\bigr)\in G$ and each $Z=a(y,Y)\in {\mathcal D}$, we have
\begin{align*}J(&g, Z)=e^{i\lambda c_0}\Det ({\bar Q}Y+{\bar P})^{-m}\\
&\times \exp \left(\tfrac{\lambda}{4}\left(z_0{\bar z_0}+2{\bar z_0}Py+yP^t{\bar Q}y-({\bar z_0}+{\bar Q}y)(PY+Q)({\bar Q}Y+{\bar P})^{-1}({\bar z_0}+{\bar Q}y\right)
\right)\\ \end{align*}
\end{enumerate} \end{proposition}

We assume now that $m+n+1/2<0$. Then the formula
\begin{equation*}\pi_{\chi}(g)f(Z)=J_{\chi}(g^{-1},Z)^{-1}\,f(g^{-1}\cdot
Z) \end{equation*}  defines a unitary representation of $G$ on
${\mathcal H}_{\chi}$, \cite{Ne}, p. 540.

Note that ${\mathcal H}_{\chi}$ is a reproducing kernel
Hilbert space. Indeed, if we set $e_Z(W):=K_{\chi}(W,Z)$ then we have we
have the reproducing property $f(Z)=\langle f,e_Z\rangle_{\chi}$ for
each $f\in {\mathcal H}_{\chi}$ and each $Z\in {\mathcal D}$
\cite{Ne}, p. 540. Here $\langle\cdot,\cdot\rangle_{\chi}$ is
the inner product on ${\mathcal H}_{\chi}$.

By using the coherent states $(e_Z)_{Z\in {\mathcal D}}$, we can define the Berezin symbol
$S_{\chi}(A)(Z)$ and the double Berezin symbol $s_{\chi}(A)(Z,W)$ of an operator $A$
on ${\mathcal H}_{\chi}$ as this was done for operators on ${\mathcal F}_{\lambda}$
in Section \ref{sec:2}. We can also verify that the kernel $k_A(Z,W)$ of $A$ satisfying
\begin{equation*}f(Z) =\int_{\mathcal D}\,k_A(Z,W)f(W)
K_{\chi}(W,W)^{-1}c_{\chi}d\mu (W)\end{equation*}
for each $f\in {\mathcal H}_{\chi}$ is given by
\begin{equation*}k_A(Z,W)=\langle A\,e_W\,,\,e_Z\rangle_{\chi}=s_{\chi}(A)(Z,W)
\langle e_W\,,\,e_Z\rangle_{\chi}. \end{equation*}

\section{Kernels of metaplectic operators} \label{sec:6}

By Proposition \ref{prop:Jac}, we have
\begin{equation*}(\pi_{\chi}((z_0,\bar{z_0}),0,I_{2n})f)(Z)=
\exp \left( \tfrac{\lambda}{4}\left(-\vert z_0\vert^2+2\bar{z_0}y+\bar{z_0}Y\bar{z_0}\right)
\right)\,f(a(y-z_0+Y\bar{z_0},Y))
\end{equation*}
for $Z=a(y,Y)\in {\mathcal D}$, $z_0\in {\mathbb C}^n$ and $f\in {\mathcal H}_{\chi}$.

Since this formula for $\pi_{\chi}((z_0,\bar{z_0}),0,I_{2n})$ is close to that of $\rho_{\lambda}
((z_0,\bar{z_0}),0)$, see Section \ref{sec:2}, a natural idea is then to use the commutation relation
\begin{equation*}\pi_{\chi}(k(z_0,\bar{z_0}),0,I_{2n})\pi_{\chi}((0,0),0,k)=\pi_{\chi}((0,0),0,k)\pi_{\chi}((z_0,\bar{z_0}),0,I_{2n})
\end{equation*}
for $k\in S$ in order to find a solution $b_k(z,w)$ of Eq. \ref{eq:relatker1}.

\begin{proposition} \label{prop:relatker2} For $k\in S$, let $B_k(Z,W)$ be the kernel
of $\pi_{\chi}((0,0),0,k)$. Then 
\begin{equation*}b_k(z,w):= B_k(a(0,z),a(0,w))\end{equation*}
 is a solution of Eq. \ref{eq:relatker1}.
\end{proposition}

\begin{proof} Let $k\in S$ and $Z=a(y,Y)\in  {\mathcal D}$. For each 
$f\in {\mathcal H}_{\chi}$, we have
\begin{equation*}(\pi_{\chi}((0,0),0,k)f)(Z)=
\int_{\mathcal D}\,B_k(Z,W)f(W)
K_{\chi}(W,W)^{-1}c_{\chi}d\mu (W).\end{equation*}

Then, on the one hand, we have
\begin{align*}(\pi_{\chi}(k(z_0,\bar{z_0}),0,I_{2n})&\pi_{\chi}((0,0),0,k)f)(Z)=
\exp \left(  \tfrac{\lambda}{4}\left(-\vert kz_0\vert^2+2(\overline{kz_0})y+
(\overline{kz_0})Y(\overline{kz_0})\right)\right)\\ 
&\times \int_{\mathcal D}\,B_k(a(y-kz_0+Y(\overline{kz_0}),Y),W)f(W)
K_{\chi}(W,W)^{-1}c_{\chi}d\mu (W). \end{align*}

On the other hand, we have
\begin{align*}&(\pi_{\chi}((0,0),0,k)\pi_{\chi}((z_0,\bar{z_0}),0,I_{2n})f)(Z)\\
&= \int_{\mathcal D}\,B_k(Z,W)(\pi_{\chi}((z_0,\bar{z_0}),0,I_{2n})f)(W)
K_{\chi}(W,W)^{-1}c_{\chi}d\mu (W)\\
&=\int_{\mathcal D}\,B_k(Z,W)
\exp \left(  \tfrac{\lambda}{4}\left(-\vert z_0\vert^2+2\bar{z_0}v+
\bar{z_0}V(\bar{z_0})\right)\right)\\
& \times
f(a(v-z_0+V\bar{z_0},V))
K_{\chi}(W,W)^{-1}c_{\chi}d\mu (W)
\end{align*}
with the notation $W=a(v,V)$. We perform in this integral the change of variables
$v\rightarrow v+z_0-V\bar{z_0}$ and we find
\begin{align*}&(\pi_{\chi}((0,0),0,k)\pi_{\chi}((z_0,\bar{z_0}),0,I_{2n})f)(Z)\\
&=\int_{\mathcal D}\,B_k(Z,a(v+z_0-V\bar{z_0},V))
\exp \left(  \tfrac{\lambda}{4}\left(\vert z_0\vert^2+2\bar{z_0}v-
\bar{z_0}V\bar{z_0}\right)\right)f(W)\\
&\times K_{\chi}(a(v+z_0-V\bar{z_0},V),a(v+z_0-V\bar{z_0},V))^{-1}c_{\chi}d\mu (W).
\end{align*}
Then, writing that
$\pi_{\chi}(k(z_0,\bar{z_0}),0,I_{2n})\pi_{\chi}((0,0),0,k)$ and $\pi_{\chi}((0,0),0,k)\pi_{\chi}((z_0,\bar{z_0}),0,I_{2n})$ have the same kernel, we obtain 
\begin{align*}
&\exp \left(  \tfrac{\lambda}{4}\left(-\vert kz_0\vert^2+2(\overline{kz_0})y+
(\overline{kz_0})Y(\overline{kz_0})\right)\right)\, B_k(a(y-kz_0+Y(\overline{kz_0}),Y),W)
K_{\chi}(W,W)^{-1}\\
&=\exp \left(  \tfrac{\lambda}{4}\left(\vert z_0\vert^2+2\bar{z_0}v-
\bar{z_0}V\bar{z_0}\right)\right)B_k(a(y,Y),a(v+z_0-V\bar{z_0},V))\\
&\times  K_{\chi}(a(v+z_0-V\bar{z_0},V),a(v+z_0-V\bar{z_0},V))^{-1}.
\end{align*}
Note that $K_{\chi}(a(v,0),a(v,0))=\exp\left(\tfrac{\lambda}{2}\vert v \vert^2\right)$.
Then, taking $Y=V=0$ in the above equality we get
\begin{align*}
&\exp \left(  \tfrac{\lambda}{4}\left(-\vert kz_0\vert^2+2(\overline{kz_0})y-2
\vert v \vert^2\right)\right)\, B_k(a(y-kz_0,0),a(v,0))\\
&=\exp \left(  \tfrac{\lambda}{4}\left(\vert z_0\vert^2+2\bar{z_0}v-
2\vert v+z_0\vert^2\right)\right)\,B_k(a(y,0),a(v+z_0,0)),
\end{align*} hence the desired result.
\end{proof}

The next step is then to compute $B_k(a(y,0),a(v,0))$.

\begin{proposition} \label{prop:exprBk} Let $Z=a(y,0)$ and $W=a(v,0)$. Then we have
\begin{equation}\label{eq:Bk} B_k(Z,W)=(\Det P)^m\exp \left(\tfrac{\lambda}{4}\left(
y({\bar Q}P^{-1}y)+2y(P^t)^{-1}{\bar v}-{\bar v}(P^{-1}Q{\bar v})\right)\right).
\end{equation}
\end{proposition}

\begin{proof} Let $Z=a(y,0)$, $W=a(v,0)$ and $g=((0,0),0,k)$ where $k\in S$.
Then we have
\begin{align*}B_k(Z,W)&=K_{\pi_{\chi}(g)}(Z,W)=\langle \pi_{\chi}(g)e_W,e_Z\rangle_{\chi}\\
&=(\pi_{\chi}(g)e_W)(Z)=J_{\chi}(g^{-1},Z)^{-1}e_W(g^{-1}\cdot Z)\\
&=J_{\chi}(g^{-1},Z)^{-1}\langle e_W,e_{g^{-1}\cdot Z}\rangle_{\chi}=J_{\chi}(g^{-1},Z)^{-1}K_{\chi}(g^{-1}\cdot Z,W).
\end{align*}
Note that $k^{-1}=\bigl(\begin{smallmatrix} P^{\ast}&-Q^t\\
-Q^{\ast}& P^t\end{smallmatrix}\bigr)$. Then we get
\begin{equation*}g^{-1}\cdot Z=a(P^{\ast}y-Q^t(P^t)^{-1}Q^{\ast}y,-Q^t(P^t)^{-1})=a(P^{-1}y,-Q^t(P^t)^{-1})=a(P^{-1}y,-P^{-1}Q)
\end{equation*}
since \begin{equation*}P^{\ast}-Q^t(P^t)^{-1}Q^{\ast}=P^{\ast}-P^{-1}QQ^{\ast}=P^{-1}(PP^{\ast}-QQ^{\ast})=P^{-1}.\end{equation*}
Thus, by Proposition \ref{prop:Jac}, we have
\begin{equation*}K_{\chi}(g^{-1}\cdot Z,W)=\exp \left( 
\tfrac{\lambda}{4}\left(2y(P^t)^{-1}{\bar v}-{\bar v}P^{-1}Q{\bar v}\right)\right).
\end{equation*}
Moreover, by Proposition \ref{prop:Jac} again, we find
\begin{equation*}J_{\chi}(g^{-1},Z)^{-1}=(\Det P)^m
\exp \left( 
\tfrac{\lambda}{4}\left(y(P^{\ast})^tQ^{\ast}y-y{\bar Q}Q^t(P^t)^{-1}Q^{\ast}y\right)\right).\end{equation*}
Since we have
\begin{equation*}(P^{\ast})^t-{\bar Q}Q^t(P^t)^{-1}=((P^{\ast})^tP^t-{\bar Q}Q^t)(P^t)^{-1}=(PP^{\ast}-QQ^{\ast})^t(P^t)^{-1}=(P^t)^{-1},
\end{equation*} the result follows.
\end{proof}

For each $k\in S$, we denote by $A_k$ the operator on ${\mathcal F}_{\lambda}$ with kernel
$$b_k(z,w)=B_k(a(z,0),a(w,0)).$$ By Schur's lemma, there exists, for each $k,k'\in S$ a scalar
$\alpha (k,k')$ such that $A_{kk'}=\alpha (k,k')A_kA_{k'}$.

\begin{lemma}\label{lem:alpha} Let $k=\bigl(\begin{smallmatrix} P&Q\\
{\bar Q}&{\bar P} \end{smallmatrix}\bigr)$, $k'=\bigl(\begin{smallmatrix} P'&Q'\\
{\bar Q'}&{\bar P'} \end{smallmatrix}\bigr)$ and $k''=\bigl(\begin{smallmatrix} P''&Q''\\
{\bar Q''}&{\bar P''} \end{smallmatrix}\bigr)$ in $S$. The we have
\begin{equation*}(\Det P'')^m=\alpha (k,k')(\Det P)^m(\Det P')^m(\Det (P^{-1}P''P'^{-1})^{-1/2}.\end{equation*}
Here, if $z\in {\mathbb C}$, we define $z^{1/2}$ as the principal determination of the square-root (with branch cut along the negative real axis).
\end{lemma}
\begin{proof} Since the kernel of $A_kA_{k'}$ is the convolution of the kernels of $A_k$
and $A_{k'}$, we have
\begin{equation*}b_{kk'}(z,w)=\alpha (k,k')\,\int_{{\mathbb C}^n}b_k(z,u)b_{k'}(u,w)
e^{-\lambda \vert
u\vert^2/2}\,d\mu_{\lambda} (u).
\end{equation*} Taking $z=w=0$, we get
\begin{equation*}(\Det P'')^m=\alpha (k,k')(\Det P)^m(\Det P')^m
\int_{{\mathbb C}^n}\exp \left( \tfrac{\lambda}{4}\left(-{\bar u}P^{-1}Q{\bar u}+u {\bar Q'}
P'^{-1}u\right)\right)e^{-\lambda \vert
u\vert^2/2}\,d\mu_{\lambda} (u).
\end{equation*}
Recall that $P^{-1}Q$ and ${\bar Q'}
P'^{-1}$ are symmetric. Then, the integral in the preceding equality can be evaluated by using
\cite{Fo}, Theorem 3, p. 258 and its value is
\begin{equation*}\Det ^{-1/2}(I_n+P^{-1}Q{\bar Q'}
P'^{-1})=\Det ^{-1/2}(P^{-1}(PP'+Q{\bar Q'})P'^{-1})=
\Det ^{-1/2}(P^{-1}P''P'^{-1}).\end{equation*}
The result follows.
\end{proof}
We are now in position to recover \cite{Fo}, Theorem (4.37) (this result is due to V. Bargmann and C. Itzykson). For $k\in S$, we denote by $\sigma(k)$ the operator $A_k$ with kernel
$b_k(z,w)$ corresponding to $m=-1/2$. Then $\sigma$ is called the metaplectic representation of $S$. Note that the value $m=-1/2$ does not correspond to a holomorphic representation of 
$G$, see (2) of Proposition \ref{prop:Jac}.

\begin{proposition} \begin{enumerate} \item For each $k, k'\in S$, we have $\sigma (kk')=\pm
\sigma(k)\sigma(k')$.
\item For each $k\in S$, $\sigma(k)$ is unitary.
\end{enumerate}
\end{proposition}

\begin{proof} (1) This is an immediate consequence of Lemma \ref{lem:alpha}.
(2) From the formula for $b_k$ (see Eq. \ref{eq:Bk} ), we deduce that we have
$b_{k^{-1}}(z,w)=\overline {b_k(z,w)}$ for each $z,w \in {\mathbb C}^n$, hence
$A_{k^{-1}}=A_k^{\ast}$. This implies that $A_kA_k^{\ast}=\pm Id$. Since
$A_kA_k^{\ast}$ is positive, we have  $A_kA_k^{\ast}= Id$. By the same argument,
we also obtain $A_k^{\ast}A_k= Id$.
\end{proof}

We can also give a formula for the differential of $\sigma$.
\begin{proposition} Let $X=\bigl(\begin{smallmatrix} A&B\\
{\bar B}&{\bar A} \end{smallmatrix}\bigr)\in {\mathfrak s}$. Then we have
\begin{align*}(d\sigma &(X)f)(z)=(-\tfrac{1}{2}\Tr(A)+\tfrac{\lambda}{4}z({\bar B}z))f(z)\\
&-\sum_{j=1}^{n}(Az)_j\frac{\partial f}{\partial z_j}-\tfrac{1}{\lambda}\sum_{j,k}b_{jk}
\frac{\partial^2f}{\partial z_j\partial z_k}
\end{align*} where $B=(b_{jk})$.
\end{proposition}

\begin{proof} By differentiating the following formula for the kernel $b_k$ of $\sigma(k)$
\begin{equation*} b_k(z,w)=(\Det P)^{-1/2}\exp \left(\tfrac{\lambda}{4}\left(
z({\bar Q}P^{-1}z)+2(P^{-1}z){\bar w}-{\bar w}(P^{-1}Q{\bar w})\right)\right),
\end{equation*}
we obtain a formula for the kernel $b_X(z,w)$ of $d\sigma(X)$:
\begin{equation*} b_X(z,w)=\left(-\tfrac{1}{2}\Tr(A)+\tfrac{\lambda}{4}z({\bar B}z)-
\tfrac{\lambda}{2}(Az){\bar w}-\tfrac{\lambda}{4}{\bar w}(B{\bar w})\right)\exp \left(\tfrac{\lambda}{2}z{\bar w}\right).\end{equation*}
Remark that by differentiating the reproducing property
\begin{equation*}f(z)=\langle f,e_z\rangle_{\mathcal F_{\lambda}}=
\int_{{\mathbb C}^n} e^{\lambda z{\bar w}/2}f(w)e^{-{\lambda} \vert
w\vert^2/2}\,d\mu_{\lambda} (w)\end{equation*} under the integral sign,
we get
\begin{equation*}\frac{\partial f}{\partial z_j}=\tfrac{\lambda}{2}
\int_{{\mathbb C}^n} {\bar w_j}e^{\lambda z{\bar w}/2}f(w)e^{-\lambda \vert
w\vert^2/2}\,d\mu_{\lambda} (w)\end{equation*} for each $j=1,2,\ldots,n$ and,
by differentiating again, we also obtain
\begin{equation*}\frac{\partial^2f}{\partial z_j\partial z_k}=\left(\tfrac{\lambda}{2}\right)^2
\int_{{\mathbb C}^n} {\bar w_j}{\bar w_k}e^{{\lambda} z{\bar w}/2}f(w)e^{-\lambda \vert
w\vert^2/2}\,d\mu_{\lambda} (w)\end{equation*} for each $j,k=1,2,\ldots,n$.
This allows us to compute
\begin{equation*}(d\sigma (X)f)(z)=\int_{{\mathbb C}^n}b_X(z,w)f(w)e^{-{\lambda} \vert
w\vert^2/2}\,d\mu_{\lambda} (w)\end{equation*} and to get the desired result.
\end{proof}

We can also give formulas for the Berezin symbols of $\sigma(k)$ for $k\in S$ and 
$d\sigma(X)$ for $X\in {\mathfrak s}$. We immediately obtain the following proposition.

\begin{proposition} \label{propSsigma} \begin{enumerate} \item Let $k=\bigl(\begin{smallmatrix} P&Q\\
{\bar Q}&{\bar P} \end{smallmatrix}\bigr)\in S$. we have 
\begin{equation*}S_{\lambda}(\sigma(k))(z)=(\Det P)^{-1/2}\exp \left(\tfrac{\lambda}{4}\left(
z({\bar Q}P^{-1}z)+2{\bar z}(P^{-1}-I_n)z-{\bar z}(P^{-1}Q{\bar z})\right)\right).
\end{equation*}
\item  Let $X=\bigl(\begin{smallmatrix} A&B\\
{\bar B}&{\bar A} \end{smallmatrix}\bigr)\in {\mathfrak s}$. Then we have
\begin{equation*}S_{\lambda}(d\sigma (X))(z)=-\tfrac{1}{2}\Tr(A)+\tfrac{\lambda}{4}z({\bar B}z)
-\tfrac{\lambda}{2}(Az){\bar z}-\tfrac{\lambda}{4}{\bar z}(B{\bar z}).\end{equation*}
\end{enumerate}
\end{proposition}

\section{Complex Weyl symbols of metaplectic operators} \label{sec:7}
In this section, we compute $W_0(\sigma(k))$ for $k\in S$ and $W_0(d\sigma(X))$ for
$X\in {\mathfrak s}$. We begin with two technical lemmas. The first one is a variant of \cite{Fo}, Theorem 3, p. 258.

\begin{lemma} \label{lemgauss} Let $A,B, D$ be $n\times n$ complex matrices such that
$A^t=A, D^t=D$. Let $M=\bigl(\begin{smallmatrix} A&B^t\\
B&D \end{smallmatrix}\bigr)$, $U=\bigl(\begin{smallmatrix} I_n&iI_n\\
I_n&-iI_n \end{smallmatrix}\bigr)$ and $N=U^tMU$. Assume that $\Rea (N)$ is positive definite. Let $u,v \in {\mathbb C}^n$. Then we have 
\begin{align*}\int_{{\mathbb C}^n}&\exp\left(-\left( w(Aw)+{\bar w}(D{\bar w})+2{\bar w}(Bw)\right)\right)\exp (uw+v{\bar w})\,dm(w)\\
=&\pi^n(\Det N)^{-1/2}\exp \left( \tfrac{1}{4}\begin{pmatrix}u&v\end{pmatrix}M^{-1}\begin{pmatrix}u\\v\end{pmatrix}\right).
\end{align*}
\end{lemma}

\begin{proof} Write $w=x+iy$ with $x, y \in {\mathbb R}^n$. Then $\begin{pmatrix}w\\{\bar w}\end{pmatrix}=U\begin{pmatrix}x\\y\end{pmatrix}$. We have
\begin{equation*}w(Aw)+{\bar w}(D{\bar w})+2{\bar w}(Bw)=(w,{\bar w})M\begin{pmatrix}w\\{\bar w}\end{pmatrix}=(x,y)N\begin{pmatrix}x\\y\end{pmatrix}
\end{equation*} and $uw+v{\bar w}= \begin{pmatrix}u&v\end{pmatrix}U
\begin{pmatrix}x\\y\end{pmatrix}$.

The result then follows from the well-known equality
\begin{equation*}\int_{{\mathbb R}^n}\exp (-xAx+zx)\,dx=(\Det A)^{-1/2}\pi ^{n/2}
\exp\left(\tfrac{1}{4}z(A^{-1}z)\right)
\end{equation*} for $z\in {\mathbb C}^n$ and $A$ a $n\times n$ symmetric complex matrix such that $\Rea(A)$ is definite positive.
\end{proof}

\begin{lemma} \label{lemmatrices} \begin{enumerate}
\item Let $a,b, p$ be $n\times n$ complex matrices such that
$\bigl(\begin{smallmatrix} -a&I_n+p^t\\
I_n+p&d \end{smallmatrix}\bigr)$ is invertible with inverse matrix $\bigl(\begin{smallmatrix} \alpha &\beta \\
\gamma & \delta  \end{smallmatrix}\bigr)$. Then we have
\begin{equation*}\begin{pmatrix} a&I_n-p^t\\
p-I_n&d \end{pmatrix}\begin{pmatrix} \alpha &\beta \\
\gamma & \delta  \end{pmatrix} \begin{pmatrix} a&p^t-I_n\\
I_n-p&d \end{pmatrix}=\begin{pmatrix}4\delta -a&3I_n-4\gamma -p^t\\
3I_n-4\beta-p&4\alpha+d \end{pmatrix}.
\end{equation*}
\item Take $a={\bar Q}P^{-1}$, $d=P^{-1}Q$ and $p=P^{-1}$ with 
$k=\bigl(\begin{smallmatrix} P&Q\\
{\bar Q}&{\bar P} \end{smallmatrix}\bigr)\in S$. Let $J=\bigl(\begin{smallmatrix} 0&I_n\\
-I_n&0 \end{smallmatrix}\bigr)$. Then we have
\begin{equation*}\tfrac{1}{2}J(k-I_{2n})(k+I_{2n})^{-1}=\begin{pmatrix}\delta &\tfrac{1}{2}I_n-\gamma\\
\tfrac{1}{2}I_n-\beta&\alpha \end{pmatrix}.\end{equation*}
\item Let $k=\bigl(\begin{smallmatrix} P&Q\\
{\bar Q}&{\bar P} \end{smallmatrix}\bigr)\in S$. Then
\begin{equation*}\Det \begin{pmatrix}-{\bar Q}P^{-1} &I_n+(P^t)^{-1}\\
I_n+P^{-1}&P^{-1}Q \end{pmatrix}=(-1)^n(\Det P)^{-1}\Det(k+I_{2n}).
\end{equation*}\end{enumerate}
\end{lemma}

\begin{proof} (1) By writing
\begin{equation*}\begin{pmatrix} \alpha &\beta \\
\gamma & \delta  \end{pmatrix} \begin{pmatrix} -a&I_n+p^t\\
I_n+p&d \end{pmatrix}=\begin{pmatrix} -a&I_n+p^t\\
I_n+p&d \end{pmatrix}\begin{pmatrix} \alpha &\beta \\
\gamma & \delta  \end{pmatrix} =I_{2n}
\end{equation*} 
we obtain the series of equations
\begin{align*}\alpha a=&\beta (I_n+p)-I_n\quad; \quad \gamma a=\delta (I_n+p);\\
\beta d=&-\alpha (I_n+p^t)\quad; \quad \delta d=I_n-\gamma (I_n+p^t);\\
a \alpha =&(I_n+p^t)\gamma -I_n\quad; \quad a\beta=(I_n+p^t)\delta;\\
d\gamma=&-(I_n+p)\alpha \quad; \quad d\delta=I_n-(I_n+p)\beta.
\end{align*}
By using these equations, we obtain firstly
\begin{equation*}\begin{pmatrix} \alpha &\beta \\
\gamma & \delta  \end{pmatrix} \begin{pmatrix} a&-I_n+p^t\\
I_n-p&d \end{pmatrix}=\begin{pmatrix} 2\beta-I_n &-2\alpha \\
2\delta & I_n-2\gamma  \end{pmatrix}\end{equation*}
and, secondly,
\begin{equation*}\begin{pmatrix} a&I_n-p^t\\
-I_n+p&d \end{pmatrix}\begin{pmatrix} 2\beta-I_n &-2\alpha \\
2\delta & I_n-2\gamma  \end{pmatrix}= 
\begin{pmatrix}4\delta -a&3I_n-4\gamma -p^t\\
3I_n-4\beta-p&4\alpha+d \end{pmatrix}.
\end{equation*}

(2)-(3) First we have
\begin{align*}
\begin{pmatrix}-{\bar Q}P^{-1} &I_n+(P^t)^{-1}\\
I_n+P^{-1}&P^{-1}Q \end{pmatrix}&\begin{pmatrix} -P&-Q\\
0&I_n \end{pmatrix}=\begin{pmatrix} {\bar Q}&{\bar Q}P^{-1}Q +I_n+(P^t)^{-1}\\
-I_n-P&-Q \end{pmatrix}\\=&\begin{pmatrix} {\bar Q}&I_n+{\bar P}\\
-I_n-P&-Q \end{pmatrix}=J(k+I_{2n})\end{align*}
since \begin{equation*}{\bar Q}P^{-1}Q +(P^t)^{-1}=({\bar Q}P^{-1}QP^t+I_n)(P^t)^{-1}=({\bar Q}Q^t+I_n)(P^t)^{-1}={\bar P}.\end{equation*}
On the one hand, passing to the determinant, we obtain (3) and, on the other hand, we 
deduce that
\begin{equation*}\begin{pmatrix} -P&-Q\\
0&I_n \end{pmatrix}=\begin{pmatrix} \alpha &\beta \\
\gamma & \delta  \end{pmatrix}\,J(k+I_{2n}).\end{equation*}
This implies that
\begin{align*}\begin{pmatrix}\delta &\tfrac{1}{2}I_n-\gamma\\
\tfrac{1}{2}I_n-\beta&\alpha \end{pmatrix}&=\frac{1}{2}\begin{pmatrix} 0&I_n\\
I_n&0 \end{pmatrix}-J\begin{pmatrix} \alpha &\beta \\
\gamma & \delta  \end{pmatrix}J\\
&=\frac{1}{2}\begin{pmatrix} 0&I_n\\
I_n&0 \end{pmatrix}-J\begin{pmatrix} -P&-Q\\
0&I_n \end{pmatrix}(k+I_{2n})^{-1}\\
&=\left( \frac{1}{2}\begin{pmatrix} 0&I_n\\
I_n&0 \end{pmatrix}(k+I_{2n})-J\begin{pmatrix} -P&-Q\\
0&I_n \end{pmatrix}\right)(k+I_{2n})^{-1}\\
&=\frac{1}{2}J(k-I_{2n})(k+I_{2n})^{-1}.
\end{align*}
\end{proof}

We denote by $\Arg(z)$ the principal argument of $z\in {\mathbb C}$.

\begin{proposition} \label{propW0sigma} Let $k=\bigl(\begin{smallmatrix} P&Q\\
{\bar Q}&{\bar P} \end{smallmatrix}\bigr)\in S$. Then we have
\begin{equation*}W_0(\sigma(k))(z)=c_n(k) \exp \left( \tfrac{\lambda}{2}\begin{pmatrix}z&
{\bar z}\end{pmatrix}J(k-I_{2n})
(k+I_{2n})^{-1}\begin{pmatrix}z\\{\bar z}\end{pmatrix}\right)\end{equation*}
where 
\begin{align*}
c_n(k)&=2^n(\Det(I_{2n}+k))^{-1/2}\, if \, \Det(I_{2n}+k)>0;\\
c_n(k)&=-i2^n\vert \Det(I_{2n}+k)\vert^{-1/2}\, if \, \Det(I_{2n}+k)<0\, and\,
\Arg(\Det(P))\in ]0,\pi[ ;\\
c_n(k)&=i2^n\vert \Det(I_{2n}+k)\vert^{-1/2}\, if \, \Det(I_{2n}+k)<0\, and \,
\Arg(\Det(P))\in ]-\pi,0[.\end{align*}
\end{proposition}

\begin{proof} Let $k=\bigl(\begin{smallmatrix} P&Q\\
{\bar Q}&{\bar P} \end{smallmatrix}\bigr)\in S$. Recall that $W_0(\sigma(k))$ is given by
\begin{equation*}W_0(\sigma(k))(z)=\left(\tfrac{\lambda}{\pi}\right)^n\int_{{\mathbb C}^n}b_k(z+w,z-w)\exp \left(\tfrac{\lambda}{2}\left(-z{\bar z}-w{\bar w}+z{\bar w}-{\bar z}w\right)\right) dm(w)\end{equation*}
where the kernel of $\sigma(k)$ is
\begin{equation*}b_k(z,w)=(\Det P)^{-1/2}\exp \left(\tfrac{\lambda}{4}\left(
z({\bar Q}P^{-1}z)+2{\bar w}(P^{-1}z)-{\bar w}(P^{-1}Q{\bar w})\right)\right).
\end{equation*}
Then we get
\begin{align*}W_0&(\sigma(k))(z)=\left(\tfrac{\lambda}{\pi}\right)^n
(\Det P)^{-1/2}\exp \left(\tfrac{\lambda}{4}\left(
z({\bar Q}P^{-1}z)+2{\bar z}(P^{-1}-I_n)z-{\bar z}(P^{-1}Q{\bar z})\right)\right)\\
\times & \int_{{\mathbb C}^n}\exp \left(\tfrac{\lambda}{4}\left(w{\bar Q}P^{-1}w-{\bar w}(P^{-1}Q{\bar w})-2{\bar w}(I_n+P^{-1})w\right)\right) \\
\times & \exp \left(\tfrac{\lambda}{2}\left(({\bar Q}P^{-1}z+((P^t)^{-1}-I_n){\bar z})w+
((I_n-P^{-1})z+P^{-1}Q{\bar z}){\bar w}\right)\right) \,dm(w)
\end{align*}
The integral $I(k)$ in the preceding formula can be evaluated by using Lemma \ref{lemgauss}
with \begin{equation*}M=\tfrac{\lambda}{4}\begin{pmatrix}-{\bar Q}P^{-1} &I_n+(P^t)^{-1}\\I_n+P^{-1}&P^{-1}Q \end{pmatrix}\end{equation*} and
\begin{equation*}u=\tfrac{\lambda}{2}({\bar Q}P^{-1}z+((P^t)^{-1}-I_n){\bar z})\,;\,v=\tfrac{\lambda}{2}((I_n-P^{-1})z+P^{-1}Q{\bar z}).\end{equation*}
Observing that, by Lemma \ref{lemmatrices}, we have
\begin{align*}
\begin{pmatrix}u&v\end{pmatrix}&M^{-1}\begin{pmatrix}u\\v\end{pmatrix}\\
&=\lambda \begin{pmatrix}z&{\bar z}\end{pmatrix}\begin{pmatrix}{\bar Q}P^{-1} &I_n-(P^t)^{-1}\\
-I_n+P^{-1}&P^{-1}Q \end{pmatrix}\begin{pmatrix}-{\bar Q}P^{-1} &I_n+(P^t)^{-1}\\
I_n+P^{-1}&P^{-1}Q \end{pmatrix}^{-1}\\
&\times \begin{pmatrix}{\bar Q}P^{-1} &-I_n+(P^t)^{-1}\\
I_n-P^{-1}&P^{-1}Q \end{pmatrix}
\begin{pmatrix}z\\{\bar z}\end{pmatrix}\\
&=\lambda \begin{pmatrix}z&{\bar z}\end{pmatrix}\begin{pmatrix}4\delta-{\bar Q}P^{-1} &3I_n-4\gamma-(P^t)^{-1}\\
3I_n-4\beta -P^{-1}&4\alpha+P^{-1}Q \end{pmatrix}\begin{pmatrix}z\\{\bar z}\end{pmatrix},
\end{align*}  we find that
\begin{equation*}I(k)=\pi^n(\Det U^tMU)^{-1/2}\exp \left(\tfrac{\lambda}{4}
\begin{pmatrix}z&{\bar z}\end{pmatrix}\begin{pmatrix}4\delta-{\bar Q}P^{-1} &3I_n-4\gamma-(P^t)^{-1}\\
3I_n-4\beta -P^{-1}&4\alpha+P^{-1}Q \end{pmatrix}\begin{pmatrix}z\\{\bar z}\end{pmatrix}\right).\end{equation*}
Then we get
\begin{equation*}W_0(\sigma(k))(z)={\lambda}^n(\Det P)^{-1/2}(\Det U^tMU)^{-1/2}
\exp \left(\lambda \begin{pmatrix}z&{\bar z}\end{pmatrix}\begin{pmatrix}\delta &\tfrac{1}{2}I_n-\gamma\\
\tfrac{1}{2}I_n-\beta &\alpha \end{pmatrix}\begin{pmatrix}z\\{\bar z}\end{pmatrix}
\right). \end{equation*}
Note that 
\begin{equation*}\Det U^tMU=(-1)^n2^{2n}\Det(M)=2^{2n}(\lambda/4)^{2n}\Det(J(k+I_
{2n}))(\Det (P))^{-1}.\end{equation*}

Finally, by Lemma  \ref{lemmatrices} again, we obtain 
\begin{equation*}W_0(\sigma(k))(z)=c_n(k) \exp \left( \tfrac{\lambda}{2}
\begin{pmatrix}z&{\bar z}\end{pmatrix}J(k-I_{2n})
(k+I_{2n})^{-1}\begin{pmatrix}z\\{\bar z}\end{pmatrix}\right)\end{equation*}
where $c_n(k)=2^n(\Det P)^{-1/2}((\Det (k+I_{2n})(\Det P)^{-1})^{-1/2}$.
The result hence follows by taking into account the fact that, since $k=U^{-1}gU$ with
$g\in Sp(n,{\mathbb R})$, we have $\Det(I_{2n}+k)=\Det(I_{2n}+g)\in {\mathbb R}$. 
\end{proof}

\begin{proposition}\label{propW0dsigma} Let $X=\bigl(\begin{smallmatrix} A&B\\
{\bar B}&{\bar A} \end{smallmatrix}\bigr)\in {\mathfrak s}$. Then we have
\begin{equation*}W_0(d\sigma(X))(z)=\tfrac{\lambda}{4}(z({\bar B}z)-{\bar z}(B{\bar z})-2(Az){\bar z}).
\end{equation*}
\end{proposition}

\begin{proof}There are different ways to prove this result. For instance, one can differentiate
$W_0(\sigma(k))(z)$ or one can use the integral formula for $W_0(d\sigma(X))$, see Proposition \ref{prop:intW}. However, the fastest method is based on the formula $W_0=
B_{\lambda}^{-1/2}S_{\lambda}$, see Section \ref{sec:3}.
Since $B_{\lambda}=\exp (\Delta/2\lambda)$ where
$\Delta=4\sum_{k=1}^n\partial^2/\partial z_k\partial {\bar z}_k$, see Section \ref{sec:2},
we have $B_{\lambda}^{-1/2}=\exp (-\tfrac{1}{\lambda}\sum_{k=1}^n\partial^2/\partial z_k\partial {\bar z}_k)$. By using the formula for $S_{\lambda}(d\sigma(X))$ given in Proposition \ref{propSsigma}, we get
\begin{equation*}-\tfrac{1}{\lambda}\sum_{k=1}^n\partial^2/\partial z_k\partial {\bar z}_k
(S_{\lambda}(d\sigma(X)))=\tfrac{1}{2}\sum_{k=1}^n(Ae_k)e_k=\tfrac{1}{2}\Tr(A)\end{equation*} hence
\begin{equation*}(B_{\lambda}^{-1/2}S_{\lambda}(d\sigma(X)))(z)=\tfrac{\lambda}{4}(z({\bar B}z)-{\bar z}(B{\bar z})-2(Az){\bar z}).
\end{equation*} 
\end{proof}

\section{Applications} \label{sec:8}
Here we recover some known results about the classical Weyl symbols of the metaplectic representation operators of $Sp(n,{\mathbb R})$ and about the computation of some star-exponentials.
We take $\lambda =1$. 

\subsection{Weyl symbols of metaplectic representation operators for $Sp(n,{\mathbb R})$}
The metaplectic representation $\sigma$ of $S$ can be translate to $Sp(n,{\mathbb R})$ as follows, see \cite{Fo}, Chapter IV. For each $g\in Sp(n,{\mathbb R})$, we define
$\sigma'(g) ={\mathcal B}^{-1}\sigma(UgU^{-1}){\mathcal B}$ where  $U=\left(\begin{smallmatrix} I_n&iI_n\\
I_n&-iI_n\end{smallmatrix}\right)$. Then we can deduce from Proposition \ref{propW0sigma}
a formula for $W_1(\sigma'(g))$,  $g\in Sp(n,{\mathbb R})$. Recall that $W_1$ is the inverse of the classical Weyl correspondence $\mathcal W$, see Section \ref{sec:2}.

\begin{proposition}\label{propW1sigma} Let $g=\left(\begin{smallmatrix} A&B\\
C&D\end{smallmatrix}\right)\in Sp(n,{\mathbb R})$. Then we have
\begin{equation*}W_1(\sigma'(g))(x,y)=c'_n(g) \exp \left( -i\begin{pmatrix}x&
y\end{pmatrix}J(g-I_{2n})
(g+I_{2n})^{-1}\begin{pmatrix}x\\y\end{pmatrix}\right)\end{equation*}
where 
\begin{align*}
c'_n(g)&=c_n(UgU^{-1})=2^n(\Det(I_{2n}+g))^{-1/2}\, if \, \Det(I_{2n}+g)>0;\\
c'_n(g)&=-i2^n\vert \Det(I_{2n}+g)\vert^{-1/2}\,if \,Det(I_{2n}+g)<0\, and \,
\Arg(\Det(A+D+i(C-B)))\in ]0,\pi[;\\
c'_n(g)&=i2^n\vert \Det(I_{2n}+g)\vert^{-1/2}\, if \,\Det(I_{2n}+g)<0 \, and
\, \Arg(\Det(A+D+i(C-B)))\in ]-\pi,0[.
\end{align*}
\end{proposition}

\begin{proof} Let $g=\left(\begin{smallmatrix} A&B\\
C&D\end{smallmatrix}\right)\in Sp(n,{\mathbb R})$ and $k=UgU^{-1}=
\left(\begin{smallmatrix} P&Q\\
{\bar Q}&{\bar P}\end{smallmatrix}\right)$. By Proposition \ref{propconnect} and Proposition \ref{propW0sigma}, we have, for each $(x,y)\in {\mathbb R}^{2n}$, 
\begin{align*}W_1(\sigma'(g))&(x,y)=W_1({\mathcal B}^{-1}\sigma(k){\mathcal B})(x,y)=
W_0(\sigma(k))(z)\\
&=
c_n(k) \exp \left( \tfrac{1}{2}\begin{pmatrix}z&
{\bar z}\end{pmatrix}J(k-I_{2n})
(k+I_{2n})^{-1}\begin{pmatrix}z\\{\bar z}\end{pmatrix}\right)
\end{align*} where $z=x+iy\in {\mathbb C}^{n}$.
Now, since $\left(\begin{smallmatrix}z\\{\bar z}\end{smallmatrix}\right)=U
\left(\begin{smallmatrix}x\\y\end{smallmatrix}\right)$ and $k=UgU^{-1}$, we obtain
\begin{align*}
 \tfrac{1}{2}\begin{pmatrix}z&
{\bar z}\end{pmatrix}J(k-I_{2n})
(k+I_{2n})^{-1}\begin{pmatrix}z\\{\bar z}\end{pmatrix}=& \tfrac{1}{2} \begin{pmatrix}x&
y\end{pmatrix}U^tJU
(g-I_{2n})
(g+I_{2n})^{-1}\begin{pmatrix}x\\y\end{pmatrix}\\
=&-i\begin{pmatrix}x&y\end{pmatrix}
J(g-I_{2n})
(g+I_{2n})^{-1}\begin{pmatrix}x\\y\end{pmatrix}.
\end{align*}
Moreover, we have $\Det(k+I_{2n})=\Det(g+I_{2n})$, $P=\tfrac{1}{2}(A+D+i(C-B))$
and the rest of the proposition is just a reformulation of the discussion on the value of $c_n(k)$, see Proposition \ref{propW0sigma}.
\end{proof}

Similar formulas involving Cayley transform can be found in \cite{CR1}, \cite{CR2}, \cite{Hil}, \cite{Gos}, \cite{Der1}. In the excellent book \cite{CR1}, it seems that a factor $2^n$ must be
added in Formula (3.71) to make it consistent with Formula (3.80).

The expression of $W_1(\sigma'(g))$ takes a more simple form when $g=\exp(X)$, where
$X\in sp(n,{\mathbb R})$.

\begin{corollary} \label{corW1exp} Let $X\in sp(n,{\mathbb R})$. Then we have 
\begin{equation*}W_1(\sigma'(\exp(X)))(x,y)=(\Det(\cosh (\tfrac{1}{2}X)))^{-1/2}
\exp \left(-i\begin{pmatrix}x&y\end{pmatrix}J\tanh(\tfrac{1}{2}X)
\begin{pmatrix}x\\y\end{pmatrix}\right).
\end{equation*}
\end{corollary}

\begin{proof} We just apply Proposition \ref{propW1sigma} to $g=\exp(X)$, 
$X\in sp(n,{\mathbb R})$. We have 
\begin{equation*}\Det(g+I_{2n})=2^{2n}\Det(\exp (\tfrac{1}{2}X))
\Det(\cosh (\tfrac{1}{2}X))
=2^{2n}\Det(\cosh (\tfrac{1}{2}X))\geq 0 \end{equation*} hence
$c'_n(g)=(\Det(\cosh (\tfrac{1}{2}X)))^{-1/2}$ and 
$(g+I_{2n})(g-I_{2n})^{-1}=\tanh(\tfrac{1}{2}X)$. The result follows.
\end{proof}

We can also deduce the computation of $W_1(\sigma'(X))$, $X\in sp(n,{\mathbb R})$,
from Proposition \ref{propW0dsigma}.

\begin{proposition}\label{propW1dsigma} Let $X=\left(\begin{smallmatrix} A&B\\
C&-A^t\end{smallmatrix}\right)\in sp(n,{\mathbb R})$. Then we have
\begin{equation*}W_1(d\sigma'(X))(x,y)=\tfrac{1}{2}i(2y(Ax)+y(By)-x(Cx))=
-\tfrac{1}{2}i\begin{pmatrix}x&y\end{pmatrix}JX\begin{pmatrix}x\\y\end{pmatrix}.
\end{equation*}
\end{proposition}

\begin{proof} Let $X=\left(\begin{smallmatrix} A&B\\
C&-A^t\end{smallmatrix}\right)\in sp(n,{\mathbb R})$. Define $Y=UXU^{-1}=
\left(\begin{smallmatrix} P&Q\\
{\bar Q}&{\bar P}\end{smallmatrix}\right)\in {\mathfrak s}$. Then we have
\begin{equation*}P=\tfrac{1}{2}(A+D+i(-B+C))\quad;\quad 
Q=\tfrac{1}{2}(A-D+i(B+C)).\end{equation*}
Now, we have
\begin{align*}W_1&(d\sigma'(X))(x,y)=W_1({\mathcal B}^{-1}d\sigma(Y){\mathcal B})(x,y)=
W_0(d\sigma(Y))(x+iy)\\
&=\tfrac{1}{4}((x+iy){\bar Q}(x+iy)-(x-iy)Q(x-iy)-2(P(x+iy))(x-iy)).
\end{align*}
By replacing $P$ and $Q$ with the expressions given above, the result follows from a tedious but easy calculation.
\end{proof}

\noindent \textit {Remark.} Let $X=\in sp(n,{\mathbb R})$. Then $M=\tfrac{1}{2}JX$ is a symmetric matrix and 
\begin{equation*}q_M(x,y)=\begin{pmatrix}x&y\end{pmatrix}M\begin{pmatrix}x\\y\end{pmatrix}
\end{equation*} is a quadratic form which is the Weyl symbol of the operator $id\sigma'(X)$.

On the other hand, by Corollary \ref{corW1exp}, we have, for each $t\in {\mathbb R}$ and each $(x,y)\in  {\mathbb R}^{2n}$,
\begin{equation*}\label{eqHo2}W_1(\exp (d\sigma'(tX)))(x,y)=(\Det(\cosh (\tfrac{1}{2}tX)))^{-1/2}
\exp \left(-i\begin{pmatrix}x&y\end{pmatrix}J\tanh(\tfrac{1}{2}tX)
\begin{pmatrix}x\\y\end{pmatrix}\right).\end{equation*}
By the identity theorem for analytic functions, we can take $t=i$ in the preceding equality.
This gives
\begin{align*}
W_1(\exp (d\sigma'(iX)))(x,y)=&(\Det(\cos (\tfrac{1}{2}X)))^{-1/2}
\exp \left(\begin{pmatrix}x&y\end{pmatrix}J\tan (\tfrac{1}{2}tX)
\begin{pmatrix}x\\y\end{pmatrix}\right)\\
=&(\Det(\cos (JM)))^{-1/2}
\exp \left(-\begin{pmatrix}x&y\end{pmatrix}J\tan (JM)
\begin{pmatrix}x\\y\end{pmatrix}\right).
\end{align*}
Then we recover a result of \cite{Ho2} about the Weyl symbol of the exponential of an operator whose Weyl symbol is a quadratic form $q_M(x,y)$.

\subsection{Star exponentials} The preceding results can be reformulated in terms of star exponentials for the Moyal star product. Let us recall that the notion of star product was introduced in \cite{BFF0} in order to interpret quantum mechanics as deformation of classical mechanics. Roughly speaking, a star product on a Poisson manifold $M$ is a formal deformation
$(f,g)\rightarrow f\ast g=f.g+t\{f,g\}+\sum_{l\geq 2}t^lC_l(f,g)$ of the pointwise multiplication of
$C^{\infty}(M)$. An important problem is then the computation of the star exponential
$\exp_{\ast}(f)=\sum_{l\geq 0}\frac{1}{l!}f^{\ast,l}$ for some functions $f$, see \cite{Arn},
\cite{Fr}. Note that such computations are usually done by solving some differential systems,
see \cite{BFF}, \cite{BM}. Here we only consider the Moyal star product on ${\mathbb R}^{2n}$ defined as follows.

Take coordinates
$(p,q)$ on ${\mathbb R}^{2n}\cong  {\mathbb R}^{n}\times  {\mathbb R}^{n}$ and let
$x=(p,q)$. Then one has $x_i=p_i$ for $1\leq i \leq n$ and $x_i=q_{i-n}$ for $n+1\leq i\leq 2n$. For $u,v\in C^{\infty}({\mathbb R}^{2n})$, define $P^0(u,v):=uv$,
\begin{equation*}P^1(u,v):=\sum_{k=1}^n\left( \frac{\partial u}{\partial p_k}
\frac{\partial v}{\partial q_k}-\frac{\partial u}{\partial q_k}
\frac{\partial v}{\partial p_k}\right)=\sum_{1\leq i,j\leq n}\Lambda^{ij}{\partial_{x_i}}u
{\partial_{x_j}}v \end{equation*}
(the Poisson brackets) and, more generally, for $l\geq 2$,
\begin{equation*}P^l(u,v):=\sum_{1\leq i_1,\ldots, i_l,j_1,\ldots,j_l\leq n}
\Lambda^{i_1j_1}\Lambda^{i_2j_2}\cdots \Lambda^{i_lj_l}\partial^l_{x_{i_1}\ldots x_{i_l}}u \,\partial^l_{x_{j_1}\ldots x_{j_l}}v.\end{equation*}

Then the Moyal product $\ast$ is the following formal deformation of the pointwise multiplication of $C^{\infty}({\mathbb R}^{2n})$
\begin{equation*}u \ast v:=\sum_{l\geq 0}\frac{t^l}{l!}P^l(u,v)\end {equation*}
where $t$ is a formal parameter.

Let us restrict $\ast$ to polynomials on ${\mathbb R}^{2n}$ (this is sufficient for our purpose) and take $t=-i/2$.
Then $\ast$ induces an associative product on the polynomials also denoted by $\ast$. 

On the other hand, the classical Weyl correspondence ${\mathcal W}$ (see Section \ref{sec:3})
can be extended to polynomials \cite{Ho1}. More precisely, if $f(p,q)=u(p)q^{\alpha}$ where $u$ is a polynomial on ${\mathbb R}^n$ then we have
\begin{equation*}\label{eq:W} ({\mathcal W}(f)\varphi )(p)=\left(i\frac {\partial }{ \partial
s}\right)^{\alpha} \left( u(p+\tfrac{1}{2}s)\,\varphi (p+s) \right) \Bigl
\vert _{s=0},\end{equation*}  see \cite{Vo}.
Hence if $f$ is a polynomial then ${\mathcal W}(f)$ is a differential operator with polynomial coefficients. Moreover, we can show that $\ast$ corresponds to the composition of operators in the Weyl quantization, that is, for each polynomials $f_1,f_2$ on ${\mathbb R}^{2n}$, 
we have ${\mathcal W}(f_1\ast f_2)={\mathcal W}(f_1){\mathcal W}(f_2)$. Equivalently, 
we also have $W_1(A_1)\ast W_1(A_2)=W_1(A_1A_2)$ for each differential operators $A_1,A_2$ with polynomial coefficients.

Let $M$ be a real, symmetric $(2n)\times (2n)$ matrix. Let $q_M$ be the quadratic form on
${\mathbb R}^{2n}$ associated with $M$. Let $X=-2JM\in sp(n,{\mathbb R})$. Then by Proposition \ref{propW1dsigma} we have $W_1(d\sigma'(X))=-iq_M$, hence
\begin{equation*}\exp_{\ast}(-iq_M)=\exp_{\ast}(W_1(d\sigma'(X)))=W_1(\exp(d\sigma'(X)))
=W_1(\sigma'(\exp(X)))\end{equation*}
and
\begin{equation*}\exp_{\ast}(-iq_M)(x,y)=
(\Det(\cosh (JM))^{-1/2}
\exp \left(i\begin{pmatrix}x&y\end{pmatrix}J\tanh (JM)
\begin{pmatrix}x\\y\end{pmatrix}\right).\end{equation*}
This is precisely the equation given in \cite{BM}, Theorem 1. In particular, for $q_M(x,y)=t(x^2+y^2)$, $t\in {\mathbb R}$, we have $JM=tJ$, $\cosh (JM)=\cos(t)I_{2n}$,
$\tanh(JM)=\tan (t)J$ and the preceding formula becomes
\begin{equation*}\exp_{\ast}(-it(x^2+y^2))=(\cos (t))^{-1/2}\exp( -i\tan (t)(x^2+y^2)).
\end{equation*}
Up to mormalization, we recover the formula of \cite{BFF}, Proposition 1.

Note that, in general, computations of star exponentials involve special functions, see
\cite{BFF}, \cite{GVS}, \cite{Fr}.

Strangely enough, it seems that the connection between the computation of the star exponential
of a quadratic form (for the Moyal product) and the computation of the Weyl symbol of the exponential of a differential operator whose Weyl symbol is a quadratic form
has not been mentioned in the literature.

\end{document}